\documentclass[12pt]{article}
\usepackage[utf8]{inputenc}
\usepackage[style=numeric]{biblatex}
\usepackage{eurosym}
\usepackage{amstext,amsthm}
\usepackage{amsmath}
\usepackage{amssymb,mathtools}
\usepackage[nointegrals]{wasysym}
\usepackage{mathrsfs}
\usepackage{hyperref}
\usepackage{graphicx}
\graphicspath{ {images/} }
\usepackage{subfig}
\usepackage{comment}
\usepackage{tikz}

\newtheorem{proposition}{Proposition}[section]
\newtheorem{corollary}[proposition]{Corollary}
\newtheorem{theorem}{Theorem}
\newtheorem{lemma}[proposition]{Lemma}
\theoremstyle{definition}
\newtheorem{definition}[proposition]{Definition}
\newtheorem{remark}[proposition]{Remark}
\newtheorem{assumption}[proposition]{Assumption}

\usepackage[normalem]{ulem}
\usepackage{color}

\DeclareMathAlphabet{\mathpzc}{OT1}{pzc}{m}{it}

\numberwithin{equation}{section}

\newcommand\unnumberedfootnote[1]{ %
	\let\temp=\thefootnote %
	\renewcommand{\thefootnote}{}%
	\footnote{#1}%rie
	\let\thefootnote=\temp%
	\addtocounter{footnote}{-1}}

\begin{document}
	\title{\LARGE Mean-field limits for non-linear Hawkes processes with inhibition on a Erd\H{o}s-R\'{e}nyi-graph}
	
	\author{{\sc by  J. Stiefel} \\[2ex]
		\emph{Albert-Ludwigs University Freiburg} } \date{\today}
	
	\maketitle
	
	\unnumberedfootnote{\emph{AMS 2000 subject classification.} {\tt 60G55}
		(Primary) {\tt, 60F05} (Secondary).}
	
	\unnumberedfootnote{\emph{Keywords and phrases.} Multivariate Hawkes process; Volterra equation; spike train; Erd\H{o}s-R\'{e}nyi-graph}
	
	\begin{abstract}
		\noindent
		We study a multivariate, non-linear Hawkes process $Z^N$ on a $q$-Erd\H{o}s-R\'{e}nyi-graph with $N$ nodes. Each vertex is either excitatory (probability $p$) or inhibitory (probability $1-p$). If $p\neq\tfrac12$, we take the mean-field limit of $Z^N$, leading to a multivariate point process $\bar Z$. We rescale the interaction intensity by $N$ and find that the limit intensity process solves a deterministic convolution equation and all components of $\bar Z$ are independent. The fluctuations around the mean field limit converge to the solution of a stochastic convolution equation. In the critical case, $p=\tfrac12$, we rescale by $N^{1/2}$ and discuss difficulties, both heuristically and numerically. 
	\end{abstract}
	
	\section{Introduction}
In \cite{Hawkes1971a, HawkesOakes1974}, Hawkes processes were introduced as self-excitatory point processes. Today, they are used in various fields of applications including seismology \cite{Ogata1988, Fox2016}, interactions in social networks \cite{Zipkin2016, lukasik-etal-2016-hawkes}, finance \cite{Bacry2015, Hawkes2018} and neuroscience \cite{pmid28234899, pmid25974542}. In the classical univariate, linear Hawkes process, the firing rate at time $t$ is a linear	function of $I_t := \sum_i \varphi(t-T_i)$, where the $T_i$'s are
previous jump times.  In this case, since rates cannot become negative, $\varphi\geq 0$ is required, leading to a self-excitatory process. 

Our main motivation to study Hawkes processes comes from the neurosciences. In a graph, vertices model neurons, whereas the (directed) edges are synapses linking the neurons. A point process indexed by the vertices models action potentials or spike trains of electrical impulses. Communication via synapses leads to correlated point processes such that each spike in one neuron influences the rate by which a neighboring vertex fires. From this point of view, the firing rate should include two important features. First, it is known that neurons cannot only excite others, but inhibition is another important factor (see e.g.\ \cite{Kandel-2012-Book}). This is why $\varphi \leq 0$ can occur as well, and consequently the firing rate at time $t$ has to be a non-linear function of $I_t$. Second, as not all neurons are connected, the firing rate at some vertex should depend only on the spikes of connected vertices. The Erd\H{o}s-R\'{e}nyi-model \cite{erdHos1960evolution} is one of the first and simplest models for random graphs, where the existence of edges between any two vertices is indicated by independent Bernoulli random variables with common probability $q$. In \cite{pfaffelhuber2021meanfield}, mean-field limits for the multivariate, non-linear Hawkes process with excitation and inhibition on a complete graph are derived. The main goal of this paper is to generalize these limit results to a $q$-Erd\H{o}s-R\'{e}nyi-Graph graph. 

Nonlinear multivariate Hawkes processes have been studied to some extent in the past decades, a summary can be found in the introduction of \cite{pfaffelhuber2021meanfield}. In a standard mean-field setting, all components of the Hawkes process share the same firing rate, see e.g. \cite{delattre2016hawkes}. But even if models include features leading to different rates at different vertices/neurons, the methods may be adapted to derive mean-field results, see e.g. \cite{DitlevsenLoecherbach2017} for a multi-class setting. Other works extend the standard firing rate of Hawkes processes by an age-dependence, i.e. the rate at some vertex depends on the time since the last spike at this vertex. It is still possible to derive mean-field limits \cite{Chevallier2017a}, central limit theorems \cite{Chevallier2017b}, even in the critical case, where excitation and inhibition are balanced \cite{erny2020conditional}.

~

In the present paper, we extend the model from \cite{pfaffelhuber2021meanfield} by a parameter $q$, which denotes the fraction of open connections/edges between neurons/vertices. We assume that each vertex/neuron is either excitatory or inhibitory, i.e.\ excites or inhibits all of its connected neighbors. We will denote by $p$ the fraction of excitatory vertices/neurons and distinguish the critical case $p=1/2$ from the non-critical one. For the latter, we obtain in Theorem~\ref{T1} a classical mean-field result, i.e.\, by rescaling the interaction intensity by $N$, a deterministic limit of the intensity (Theorem \ref{T1}.1) and independent point processes driven by this intensity (Theorem \ref{T1}.2) arises. We also provide a central limit result for the intensity (Theorem \ref{T1}.3). In the critical case, the methods from \cite{pfaffelhuber2021meanfield} cannot be applied. We describe the difficulties and present simulation results to visualize features of possible mean-field limits.

~

In many fields of applications, the requirement of a common connection probability in the Erd\H{o}s-R\'{e}nyi model is too stringent, as edges or vertices may have heterogeneous attributes. Some examples of Hawkes processes on more complex graphs deal with estimation of the model parameters, \cite{mei2017neural}, \cite{passino2021mutually}, \cite{verma2021self}, or perform simulations, \cite{zhou2013learning}, both for a fixed size of the graph. In \cite{agathenerine2022multivariate} and \cite{agathenerine2022longterm}, the author derives results for Hawkes processes on imhomogenuous random graphs in the mean-field setting, i.e. when the size of the graph tends to infinity, and studies the large time behaviour of the limit system. After allowing for inhibition (as described in section~6 in \cite{agathenerine2022multivariate}), these works generalize Theorem~\ref{T1}.1 in the present paper. 

~

In this work we focus on the Erd\H{o}s-R\'{e}nyi model, as it allows to derive rigorous mathematical results in the mean-field setting including a central limit theorem. More precisely, we show that the fluctuation around the mean field limit can be divided up into two parts: A common fluctuation, which is present at any vertex, plus a vertex-specific part, which is independent over the vertices and independent of the common fluctuation (see equation~\eqref{eq:flucLim}). Of course, the generalisation of these results to more realistic, and consequently more complex graph models is an interesting topic for future research.

\section{Model and assumptions}
We use the following general model for a non-linear Hawkes process:

\begin{definition}[Multi-variate, non-linear Hawkes process]
	\label{def:hawkes}
	Let $\mathbb G = (\mathbb V, \mathbb E)$ be some finite, directed graph, and write $ji \in \mathbb E$ if $j\to i$ is an edge in $\mathbb G$. 
	Consider a family of measurable, real-valued functions
	$(\varphi_{ji})_{ji \in \mathbb E}$ and a family of real-valued, non-negative functions
	$(h_i)_{i\in \mathbb V}$.  Then, a  point process $Z = (Z^i)_{i\in\mathbb V}$ (with state space $\mathbb N_0^{\mathbb V}$) is a multi-variate
	non-linear Hawkes process with interaction kernels $(\varphi_{ji})_{ji \in \mathbb E}$ and transfer functions $(h_i)_{i\in\mathbb V}$, if $Z^i, Z^j$ do not jump simultaneously for $i\neq j$, almost surely, and the compensator of $Z^i$ has the form $(\int_0^t \lambda_s^i ds)_{t\geq 0}$ with
	$$ \lambda_t^i := \lambda_t^i(Z_{s<t})
	:= h_i\Big(\sum_{j: ji \in \mathbb E} \int_0^{t-}
	\varphi_{ji}(t-s)dZ_s^j\Big), \qquad i\in\mathbb V.$$ 
\end{definition}

\noindent
We need some mimimal conditions such that the multivariate, non-linear Hawkes process is well-defined (i.e.\ exists). As mentioned in \cite{delattre2016hawkes}, Remark~5, the law of the non-linear Hawkes process is well-defined, provided that the following assumption holds. 

\begin{assumption}\label{ass:basic}
	All interaction kernels $\varphi_{ji}, ji \in \mathbb E$ are locally integrable, and all transfer functions $(h_i)_{i\in\mathbb V}$ are Lipschitz continuous. 
\end{assumption}

\begin{remark}[Interpretation and initial condition]
	\begin{enumerate}
		\item   If $dZ_s^j = 1$, we call $\varphi_{ji}(t-s) dZ_s^j$ the influence of the point at time $s$ in vertex $j$ on vertex $i$.
		\item Consider the case of monotonically increasing  transfer functions. If $\varphi_{ji} \leq 0$, we then say that vertex $j$ inhibits $i$, since any point in $Z^j$ decreases the jump rate of $Z^i$. Otherwise, if $\varphi_{ji}\geq 0$, we say that $j$ excites $i$.
		\item   In our formulation, we have $Z_0^i = 0, i\in\mathbb V$,     with the consequence that the $dZ_u^j$-integral in \eqref{eq:Z_ref} could also be extended to $-\infty$ without any     change. We note that it would also be possible to use some     initial condition, i.e.\ some (fixed) $(Z_t^i)_{i\in \mathbb     V, t\leq 0}$, and extend the integral to the negative reals. 
	\end{enumerate}
\end{remark}

~

\noindent
Let us now come to the mean-field model, where we fix some basic assumptions. Note that we will show convergence for large graphs, i.e.\ all processes come with a scaling parameter $N$, which determines the size of the graph.

\begin{assumption}[Mean-field setting]\label{ass1}
	Let
	\begin{enumerate}
		\item $\mathbb G_N = (\mathbb V_N, \mathbb E_N)$ be the $q$-Erd\H{o}s-R\'{e}nyi graph on $N$ vertices, i.e.\ $\mathbb V_N = \{1,...,N\}$, and for independent Bernoulli random variables $(V_{ji})_{j,i\in\mathbb V_N}$ with parameter $q\in[0,1]$, $ji \in\mathbb E_N$ if and only if $V_{ji}=1$;
		\item $h_i = h$ for all $i\in\mathbb V_N$ where $h \geq 0$ is
		bounded, $h$ and $\sqrt{h}$ are Lipschitz with constant $h_{Lip}$;
		\item $\varphi_{ji} = \theta_N U_j \varphi$ for all
		$j,i \in \mathbb V_N$, where $\theta_N \in \mathbb R$ and 
		\begin{itemize}
			\item $U_1, U_2,...$ are iid with
			$\mathbb P(U_1 = 1) = 1-\mathbb P(U_1 = -1) = p$,
			%\item $(V_{ji})_{ji}$ are iid with $\mathbb P(V_{ji} = 1) = 1-\mathbb P(V_{ij} = 0) = q$,
			\item $\varphi \in \mathcal C_b^1([0,\infty))$, the set of bounded continuously differentiable functions with bounded derivative.
		\end{itemize}
	\end{enumerate}
\end{assumption}

\noindent
The form of $\varphi_{ji}$ implies that node~$j$ is exciting all other
nodes with probability $p$, and inhibiting all other nodes with
probability $1-p$. Additionally $V_{ji}$ indicates if there actually is a connection from node $j$ to node $i$. Assumption \ref{ass1} leads to the intensity
$$\lambda_t^i = h\Big(\theta_N\sum_{j=1}^N \int_0^{t-} U_jV_{ji}\varphi(t-s)dZ_s^j\Big)$$
at node $i\in\mathbb V_N$.

\section{Results on the mean-field model}
Our main goal is to give a limit result on the family $Z^{N,i}$, the multivariate, non-linear Hawkes process on the graph $\mathbb G_N$ with interaction kernels and transfer functions as given in Assumption~\ref{ass1}. 

\subsection{The non-critical case}
In the case $p\neq \tfrac 12$ the limit compensator is given by $(\int_0^t h(I_s)ds)_{t\geq 0}$, where $I$ is the weak limit of $I^{N,i}$ (see \eqref{eq:IN} and \eqref{eq:IE}). We have that $(I_t)_{t\geq 0}$ follows a linear, deterministic convolution equation, and all components of the limit of $Z^{N,i}$ are independent (Theorem~\ref{T1}.2). The fluctuation around the limit converges to a stochastic convolution equation, and the correlation of the limiting fluctuations at different vertices depends on the connectivity $q$ of the graph (Theorem~\ref{T1}.3). Below, we denote by $\Rightarrow$ the weak convergence in $\mathcal D_{\mathbb R^n}([0,\infty))$, the space of cadlag paths, which is equipped with the Skorohod topology; see e.g.\ Chapter~3 in \cite{EthierKurtz86}. The proof of the following result can be found in Section~\ref{ss:proofT1}.

\begin{theorem}[Mean-field limit of multi-variate non-linear Hawkes processes, $p\neq \tfrac 12$] \label{T1} Let Assumption~\ref{ass1} hold with $p\neq \tfrac12$ and $\theta_N = \tfrac 1N$. Let $Z^N = (Z^{N,1}, ..., Z^{N,N})$ be the multivariate, non-linear Hawkes process from Definition~\ref{def:hawkes}, and
	\begin{align}\label{eq:IN}
	I^{N,i}_t := \frac 1N \sum_{j=1}^N\int_0^{t-} U_jV_{ji} \varphi(t-s) dZ_s^{N,j}.
	\end{align}
	\begin{enumerate}
		\item Then, $I^{N,i} \xrightarrow{N\to\infty} I$ almost surely, uniformly on compact time intervals and uniformly in $i\leq N$, where $I = (I_t)_{t\geq 0}$ is the unique
		solution of the integral equation
		\begin{align}
		\label{eq:IE}
		I_t = (2p-1)q \int_0^t \varphi(t-s) h(I_s)ds.
		\end{align}
		\item For all $n=1,2,...$,
		$(Z^{N,1},...,Z^{N,n}) \xRightarrow{N\to\infty} (\bar Z^1,...,\bar
		Z^n),$ where $\bar Z^1, ..., \bar Z^n$ are independent and $\bar Z^i$ is a simple point process with intensity at time $t$
		given by $h(I_t)$, $i=1,...,n$. It is possible to build
		$\bar Z^1, ..., \bar Z^n$, such that the convergence is almost surely (in Skorohod-distance).
		\item Assume that $h \in \mathcal C^1(\mathbb R)$ and that $h'$ is bounded and Lipschitz. Define $K^{N,k} = \sqrt{N} (I^{N,k} - I)$, the fluctuation around the limit at vertex $k$ as well as the mean fluctuation, $\overline K \,^N := \frac 1N \sum_{i=1}^N K^{N,i}$. Then, for all $n=1,2,...$,
		\begin{align*}
		\Big(\overline K^{N},K^{N,1},...,K^{N,n}\Big)\xRightarrow{N\to\infty} \Big(\overline K, K^{1},..., K^{n}\Big)
		\end{align*}
		where $\overline K_t=\int_0^t\varphi(t-s)d\overline G_s$, $K^k_t=\int_0^t\varphi(t-s)dG^k_s$ and
		\begin{equation}\label{eq:flucLim}
		\begin{aligned}
		\overline G_t &= q\Big(\int_0^t W h(I_s) + (2p-1)h'(I_s)\overline{K}_s ds + \int_0^t \sqrt{h(I_s)} dB_s\Big), \\
		G^{k}_t &= \overline{G}_t +  \sqrt{q(1-q)}\Big(\int_0^t \widetilde W^{k}h(I_s) ds + \int_0^t \sqrt{h(I_s)} d\widetilde B^k_s\Big).
		\end{aligned}
		\end{equation}
		Here, $B, \widetilde B^1,..., \widetilde B^n$ are independent Brownian motions, and $W, \widetilde W^1,..., \widetilde W^n$ are independent normally distributed random variables, $W\sim N(0,4p(1-p))$, $\widetilde W^i\sim N(0,1)$ for each $i=1,...,n$.
	\end{enumerate}
\end{theorem}

\begin{remark}
	This generalizes \cite[Theorem 1]{pfaffelhuber2021meanfield}. There it holds that $q=1$ and the intensities at different vertices $I^{N,i}$ are the same. Obviously 1. and 2. coincide with 1. and 2. from \cite[Theorem 1]{pfaffelhuber2021meanfield}. In 3., observe that the fluctuations around $I$ are the same at different vertices $k$, $K^{N,k}=\overline K^N$, for each $k\in\mathbb V$. In \eqref{eq:flucLim} we have $q(1-q)=0$, whence, for any $k\in\mathbb V$, 
	\begin{align*}
	G^k_t =\overline G_t= \int_0^t  Wh(I_s) + (2p-1)h'(I_s)K^k_s ds + \int_0^t \sqrt{h(I_s)} dB_s.
	\end{align*}
	This is equation (3.3) in \cite{pfaffelhuber2021meanfield}.
\end{remark}

\begin{remark}
	The Brownian motions $B, B^1,..., B^n$ arise as limit of a rescaled sum over compensated point processes, while the normal distributions $W, \widetilde W^1,..., \widetilde W^n$ arise as limit of rescaled sums over the synaptic weights $(U_j)_j, (V_{ji})_{ji}$. See sections \ref{ss:reformulation} and \ref{ss:weights} for details.
\end{remark}

\begin{remark}
	The form of \eqref{eq:IE} tells us that $I$ follows a linear
	Volterra convolution equation, \cite{berger1980volterra}. Turning into a
	differential equation, we write, using Fubini,
	\begin{align*}
	\frac{dI}{dt} & = \frac{d}{dt}(2p-1)q \Big(\int_0^t \int_s^t \varphi'(r-s) h(I_s) dr ds + \int_0^t \varphi(0) h(I_s)ds\Big)
	% \\ & = \frac{d}{dt}(2p-1) \Big(\int_0^t \int_0^r \varphi'(r-s)
	% h(I_s) ds dr + \int_0^t \varphi(0) h(I_s)ds\Big)
	\\ & 
	= (2p-1)q \Big(\int_0^t \varphi'(t-s) h(I_s) ds + \varphi(0) h(I_t) \Big).
	\end{align*}
	In particular, the special choice of $\varphi(s) = e^{-\lambda s}$ gives
	\begin{align*}
	\frac{dI}{dt} & = -\lambda (2p-1)q \Big(\int_0^t \varphi(t-s) h(I_s) ds + h(I_t) \Big) = -\lambda I_t + (2p-1)qh(I_t)),
	\end{align*}
	i.e.\ $I$ follows some ordinary differential equation in this case.
\end{remark}

\noindent
While Theorem~\ref{T1} is concerned with convergence of the limit intensity of the multivariate, non-linear Hawkes process, we are also in the situation to study convergence of the average intensity of $Z^{N,i}$. The proof of the next corollary is found in Section~\ref{ss:proofT1}.

\begin{corollary}\label{cor1}
	Let $Z^N = (Z^{N,1},...,Z^{N,N})$ be as in Theorem~\ref{T1}, and $\bar Z = (\bar Z^1, \bar Z^2,...)$ be as in Theorem~\ref{T1}.2. Then, 
	\begin{align}\label{eq:cor11}
	\frac 1N \sum_{j=1}^N \Big(Z^{N,j} - \int_0^. h(I^{N,j}_s)ds\Big) \xrightarrow{N\to\infty} 0
	\intertext{and}\label{eq:cor12}
	\frac 1N \sum_{j=1}^N \Big(Z^{N,j} - \bar Z^j\Big) \xrightarrow{N\to\infty} 0
	\end{align}
	in probability, uniformly on compact intervals. Moreover, 
	\begin{align}\label{eq:cor13}
	\frac{1}{\sqrt N} \sum_{j=1}^N \Big(Z^{N,j} - \int_0^. h(I^{N,j}_s)ds\Big) \xRightarrow{N\to\infty} \int_0^. \sqrt{h(I_s)}dB^0_s
	\end{align}
	for some Brownian motion $B^0$, and  (writing $h(I) := (h(I_t))_{t\geq 0}$)
	\begin{align}\label{eq:cor14}
	\sqrt{N} \big( h(I^{N,i}) - h(I)\big) \xRightarrow{N\to\infty} h'(I) K^i.
	\end{align}
	Let $B^U, B^0$ be correlated Brownian motions with $\mathbf E[B^U_t B^0_t] = (2p-1)t, t\geq 0$. Use $B^U$ in the definition of $\overline K$ in Theorem~\ref{T1}.3. Then, 
	\begin{align}\label{eq:cor15}
	\frac{1}{\sqrt N} \sum_{j=1}^N \Big(Z^{N,j} - \int_0^. h(I_s)ds\Big) \xRightarrow{N\to\infty} \int_0^. h'(I_s)\overline K_s ds + \int_0^. \sqrt{h(I_s)}dB^0_s.
	\end{align}
\end{corollary}

\begin{remark}[Correlation between $B^0$ and $B$]
	Let us briefly discuss the correlated Brownian motions appearing in \eqref{eq:cor15}. Clearly, the left hand side of \eqref{eq:cor15} can be built from the left hand sides of \eqref{eq:cor13} and \eqref{eq:cor14} by taking the mean value over the vertices $i$ in \eqref{eq:cor14}. The limits  $\int_0^. \sqrt{h(I_s)}d\widetilde B^0_s$ and $\overline K$ appearing on the right hand sides of \eqref{eq:cor13} and \eqref{eq:cor14} are weak limits of sums of compensated point processes. While in \eqref{eq:cor13}, we sum over all point processes in the system, $I^{N,i}$ in \eqref{eq:cor14} distinguishes between nodes with different signs $U_j$. Hence, the correlation is positive for the proportion $p$ of point processes with positive sign, and negative for the proportion $1-p$ of point processes with negative sign, summing to $p - (1-p) = 2p-1$. For more details, see Lemma \ref{l:donsker}.
\end{remark}

\subsection{The critical case}
For $p=\frac 12$, excitation and inhibition are balanced. If we rescale the interaction kernels with $\theta_N=\frac 1N$ as in Theorem~\ref{T1}, we can read off 
$$\frac 1N \sum_{j=1}^N\int_0^{t-} U_jV_{ji} \varphi(t-s) dZ_s^{j,N} \to 0$$
from Theorem~\ref{T1}.1, and the limiting point processes $\bar Z^i$ have constant intensity $h(0)$ (Theorem~\ref{T1}.2). In this critical case is natural to upscale by $\sqrt{N}$ in order to obtain non-trivial limits. This has been done under assumption \ref{ass1}, but with $q=1$, in \cite[Theorem 2]{pfaffelhuber2021meanfield}, and for a similar model in \cite{erny2020conditional}, \cite{erny2020mean}. After upscaling, the intensity at vertex $i$ and time $t$ is given by $h(I^{N,i}_t)$, where 
\begin{align*}
I^{N,i}_t = \frac 1{\sqrt{N}} \sum_{j=1}^N\int_0^{t-} U_jV_{ji} \varphi(t-s) dZ_s^{N,j}. 
\end{align*}
In order to obtain limit results we compensate $dZ_t^{N,j}$ by $h(I^{N,j})_tdt$ and apply a martingale central limit theorem. More precisely
\begin{equation}\label{eq:critical}
\begin{aligned}
I^{N,i}_t &= \frac 1{\sqrt{N}} \sum_{j=1}^N U_jV_{ji} \Big(\int_0^{t-} \varphi(t-s) dZ_s^{N,j}-\int_0^{t} \varphi(t-s)h(I^{N,j}_s) ds\Big)
\\ &+\frac 1{\sqrt{N}} \sum_{j=1}^N U_jV_{ji}\int_0^{t} \varphi(t-s)h(I^{N,j}_s) ds.
\end{aligned} 
\end{equation}
The predictable quadratic covariation of the first line is given by $\frac 1N \sum_{j=1}^NV_{ji}\int_0^{t} \varphi(t-s)h(I^{N,j}_s) ds$, and we a-priori need a limit result on this covariation as well as the second line in \eqref{eq:critical} to derive a limit of $I^{N,i}$. 
This is feasible in \cite[Theorem 2]{pfaffelhuber2021meanfield}, where $q=1$ and therefore $V_{ji}=1$, i.e. the intensities $I^{N,i}$ at different vertices $i$ coincide. In contrast, \cite{erny2020conditional} derive a mean-field limit for Hawkes processes with different intensities at different vertices. A significant difference to our model is that the influence of vertex $j$ on all other vertices, $U_j$, is not fixed over time, i.e. whenever $dZ_t^{N,j}=1$, $U_j(t)$ are centered random variables, independent over the jump times $t$. Consequently, the process \[\frac 1{\sqrt{N}} \sum_{j=1}^N \int_0^{t-}U_j(s)dZ^{N,j}_s \quad\text{is a martingale with covariation}\quad \frac {\sigma^2}N \sum_{j=1}^N\int_0^{t} h(I^{N,j}_s) ds,\]where $\sigma^2$ is the second moment of $U_j$. It suffices to a-priori derive a result on the limit of the empirical distributions of intensities at different vertices, where the exchangeability of the system can be used, to derive a limit of $I^{N,i}$. Compared to this approach, we face two major difficulties in our model: First, in the covariation of the compensated process we have to deal with the empirical distribution over the subset of vertices connected to vertex $i$. Here, one could first try to focus on the average over different vertices $i$ to replace $V_{ji}$ by its mean $q$. Second, and more challenging, we would need an a-priori result on the compensator $\frac 1{\sqrt{N}} \sum_{j=1}^N U_jV_{ji}dh(I^{N,j}_t)$, i.e. an a-priori CLT-type result for the intensities at different vertices. \\
\textbf{Simulation}\\
We may simulate a multivariate Hawkes-process on a finite, but large graph using Lewis' thinning algorithm \cite{lewis1979simulation}, \cite{ogata1981lewis}. Assume the graph consists of $N=500$ neurons and choose the parameters $p=q=0.5$. For simplicity, choose an exponential interaction kernel, $\varphi(t)=e^{-\lambda t}$ for some $\lambda>0$, and transfer function $h(x)=1+\tfrac 2\pi \arctan(x)$. Note that $\varphi$ and $h$ satisfy assumption \ref{ass1}. As the interaction kernel is exponential, the intensity at vertex is given by 
\begin{align*}
I^{N,i}_t &= \frac 1{\sqrt{N}} \sum_{j=1}^N U_jV_{ji} Z_{t-}^{N,j} - \lambda \int_0^t I^{N,i}_s ds
\\ &= \frac 1{\sqrt{N}} \sum_{j=1}^N U_jV_{ji} \Big(Z_{t-}^{N,j} - \int_0^{t} h(I^{N,j}_s) ds\Big) + \frac 1{\sqrt{N}} \sum_{j=1}^N U_jV_{ji}\int_0^{t}h(I^{N,j}_s) ds - \lambda \int_0^t I^{N,i}_s ds.
\end{align*}
We split up the martingale part,
\begin{align*}
M^i &:= \frac 1{\sqrt{N}} \sum_{j=1}^N U_jV_{ji} \Big(Z_{t-}^{N,j} - \int_0^{t} h(I^{N,j}_s) ds\Big)
\\ &= \frac 1{\sqrt{N}} \sum_{j=1}^N U_j(V_{ji}-q) \Big(Z_{t-}^{N,j} - \int_0^{t} h(I^{N,j}_s) ds\Big) + q\frac 1{\sqrt{N}} \sum_{j=1}^N U_j\Big(Z_{t-}^{N,j} - \int_0^{t} h(I^{N,j}_s) ds\Big) 
\\ &=: \widetilde M^i + qM.
\end{align*}
%\begin{figure}[b]
%	\centering
%	\includegraphics[width=\columnwidth]{../covariance_ALEA.png}
%	\caption{Approximation of the covariance $\langle M^k,M^k \rangle$ in \eqref{eq:covappr}. From left to right: $\frac 1N \sum_{j=1}^N (V_{jk}-q)^2 h(I^{N,j})$, $\frac 1N \sum_{j=1}^N (V_{jk}-q)^2\overline{h(I)}$, $1_{k=l}q(1-q)\overline{h(I)}$.}
%	\label{fig:covariance}
%\end{figure}
As the jump size of the martingale part tends to zero, the convergence is determined by the predictable covariation process. Our simulations suggest that
\begin{align}\label{eq:covappr}
\langle \widetilde M^k,\widetilde M^l \rangle &= \frac 1N \sum_{j=1}^N (V_{jk}-q)(V_{jl}-q)h(I^{N,j}) \approx \frac 1N \sum_{j=1}^N (V_{jk}-q)(V_{jl}-q)\overline{h(I)} \approx 1_{k=l}q(1-q)\overline{h(I)}, \\
\langle \widetilde M^k,M \rangle &= \frac 1N \sum_{j=1}^N (V_{jk}-q)h(I^{N,j}) \approx \frac 1N \sum_{j=1}^N (V_{jk}-q)\overline{h(I)} \approx 0, \notag \\
\langle M,M \rangle &\approx \overline{h(I)}, \notag
\end{align}
where $\overline{h(I)}=\frac 1N\sum_{j=1}^Nh(I^{N,j})$. 
%the approximation in the first line is visualized in figure \ref{fig:covariance}. 
After an application of a standard martingale central limit theorem, this would imply that
\begin{align}\label{eq:simulation}
\widetilde M^i_t &\approx \sqrt{q(1-q)}\int_0^t\sqrt{\overline{h(I_s)}}dB^i_s, \quad\text{ as well as }\quad M \approx \int_0^t\sqrt{\overline{h(I_s)}}dB_s
\end{align}
where $B, B^1,...,B^N$ are independent Brownian motions and $\overline{h(I)}=\frac 1N\sum_{j=1}^Nh(I^{N,j})$. This should be compared to Theorem~\ref{T1}.3 and lemma~\ref{l:donsker}, where we obtain a similar convergence result. There, note that $\overline{h(I)}\approx h(I)$.

\noindent
\begin{figure}[h]
	\centering
	\subfloat[Drift\label{fig:drift}]{\includegraphics[width=0.45\columnwidth]{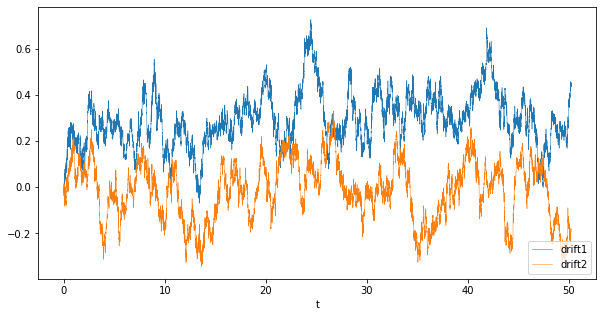}}
	\qquad
	\subfloat[Martingale part\label{fig:marti}]{\includegraphics[width=0.45\columnwidth]{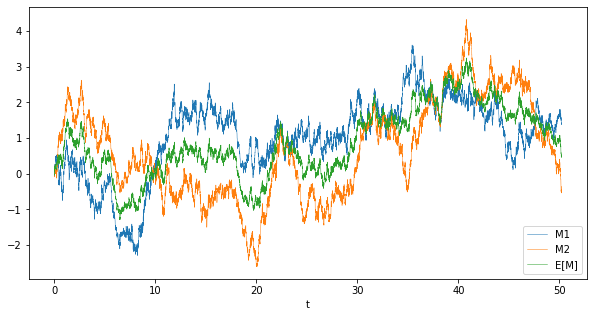}}
	\caption{Drift and martingale part of the intensity at vertex 1, 2, respectively.}
\end{figure}
As described above, the drift $\frac 1{\sqrt{N}} \sum_{j=1}^N U_jV_{ji}\int_0^{t}h(I^{N,j}_s) ds$ is more difficult to analyse, even numerically. We expect that the value of the drift depends on the configuration of the graph in some way. To see how it depends on the rescaled input of the graph to node $i$, $W^{N,i} = \frac 1{\sqrt{N}} \sum_{j=1}^N U_jV_{ji}$, manually set 
\[\sum_{j=1}^N V_{j1} = \sum_{j=1}^N V_{j2} = \frac N2, \sum_{j=1}^N V_{j1}V_{j2}=0,\]
i.e. the vertices $1, 2$ receive input from complementary parts of the graph. Then choose $U_1,...,U_N=\pm 1$, such that $\sum_j U_j\approx 0$ and $W^{N,1}=W^{N,2}$, i.e. excitation and inhibition are balanced and the mean input to node 0 and 1 are the same. As precicted above, the difference of the martingale parts $M^1$ and $M^2$, respectively, from the mean martingale part are negatively correlated , see figure \ref{fig:marti}, as $\frac 1N \sum_{j=1}^N (V_{j1}-q)(V_{j2}-q)=-q(1-q)$.

But the drifts at vertices 1 and 2 significantly differ from each other, figure \ref{fig:drift}. We expect the drift to depend on higher powers of the adjacency matrix $V=(V_{ij})_{j,i=1,...,N}$, together with multiplication with $U=(U_i)_{i=1,...,N}$ and applications of the necessarily non-linear transfer function $h$. A precise analysis of this dependence in the mean-field setting $N\to\infty$ is very complex.   

\section{Proofs}
We start off in Subsection~\ref{ss:conv} with a result on convolution equations. We proceed in~\ref{ss:reformulation} with a reformulation of the multivariate linear Hawkes process using a time-change equation, and in~\ref{ss:weights} with results on different types of mean values of the synaptic weights. Then, we prove Theorem~\ref{T1} in Subsection~\ref{ss:proofT1}.

\subsection{Convolution lemma}
\label{ss:conv}\noindent
In order to bound the value of a convolution equation by its integrator we need the following 
\begin{lemma}\label{l:conv}
	Let $J$ be the sum of an It\^{o} process with bounded coefficients and a c\`{a}dl\`{a}g pure-jump process, and let $\varphi \in \mathcal C_b^1([0,\infty))$. Then
	\begin{align}
	\label{eq:conv}
	\sup_{0\leq s\leq t}\Big(\int_0^s \varphi(s-r)dJ_r\Big)^2\leq C(\varphi,t) \cdot \sup_{0\leq s\leq t} J_s^2.
	\end{align}
\end{lemma}

\begin{proof}
	Wlog, we have $J_0=0$. By \cite[Theorem 4.A]{berger1980volterra} and Fubini's theorem for Lebesgue-integrals we can apply the Stochastic Fubini Theorem to $J$, hence
	\begin{align*}
	\sup_{0\leq s\leq t}\Big(\int_0^s \varphi(s-r)
	dJ_r\Big)^2 & \notag = \sup_{0\leq s\leq t}\Big(\int_0^s \Big( \varphi(0) + \int_r^s \varphi'(s-u) du \Big) dJ_r\Big)^2 \\
	&  \notag\leq 2\,\Big(||\varphi|| \cdot \sup_{0\leq s\leq t}
	J_s^2  + \sup_{0\leq s\leq
		t}\Big(\int_0^s \int_0^{u} \varphi'(s-u) dJ_r
	du\Big)^2\Big) \\ & \notag= 2\,\Big(||\varphi|| \cdot \sup_{0\leq s\leq t} J_r^2 + \sup_{0\leq s\leq t}\Big(\int_0^s
	\varphi'(s-u) J_u du\Big)^2\Big)
	\\ & \leq 2\,\big(||\varphi||
	+ t^2||\varphi'||^2\big) \cdot \sup_{0\leq s\leq t} J_s^2,
	\end{align*}
	where $||.||$ denotes the supremum norm.
\end{proof}
\noindent

\subsection{Reformulation of Hawkes processes}
\label{ss:reformulation}
Alternative descriptions of non-linear Hawkes processes have been given in the literature. Above all, the construction using a Poisson random measures is widely used; see e.g.\ Proposition~3 in \cite{delattre2016hawkes}. Here, we rely on the following construction using time-change equations (see e.g.\ Chapter~6 of \cite{EthierKurtz86}), which we give here without proof. 

\begin{lemma}\label{l:reform}
	Let $\mathbb G = (\mathbb V, \mathbb E)$, $(\varphi_{ji})_{ji\in \mathbb E}$ and $(h_i)_{i\in\mathbb V}$ be as in Definition~\ref{def:hawkes}, and let Assumption \ref{ass:basic} hold. A point process $Z = (Z^i)_{i\in\mathbb V}$ is a multivariate, non-linear Hawkes process with with interaction kernels $(\varphi_{ji})_{ji\in \mathbb E}$ and transfer functions $(h_i)_{i\in\mathbb V}$, if and only if it is the weak solution of the time-change equations
	\begin{align}\label{eq:Z_ref}
	Z_t^i = Y_i\Big( \int_0^t \lambda_s^i ds\Big) = Y_i\Big( \int_0^t h_i\Big(\sum_{j: ji \in \mathbb E} \int_0^{s-}
	\varphi_{ji}(s-u)dZ_u^j\Big) ds \Big),
	\end{align}
	where $(Y_i)_{i\in \mathbb V}$ is a family of independent unit rate Poisson processes. 
\end{lemma}

\noindent
Under Assumption~\ref{ass1}, the time-change equations from this lemma read, with independent unit rate Poisson processes $Y_1,...,Y_N$,
\begin{align}
Z_t^{N,i} & = Y_i\Big(\int_0^t h\Big(\sum_{j = 1}^N\int_0^s \theta_N U_jV_{ji}\varphi(s-u)dZ_u^{N,j}\Big)ds\Big).
\intertext{We rewrite this as}
\label{eq:timechange0}
Z_t^{N,i} & = Y_i\Big(\int_0^t  h(I_s^{N,i}) ds\Big)
\intertext{with}
\label{eq:JN} 
I_s^{N,i} & = \int_0^s \varphi(s-u) dJ_u^{N,i} \quad \text{ and } \quad J_u^{N,i} = \theta_N \sum_{j=1}^N U_j V_{ji} Z_u^{N,j} = \theta_N \sum_{j=1}^N U_jV_{ji}Y_j\Big(\int_0^u
h(I_s^{N,j})ds\Big).
\end{align}
To obtain convergence results, we introduce the compensated point processes
\begin{align*}
X^{N,i}_t &:= \sum_{j=1}^N U_jV_{ji}\Big(Y_j\Big(\int_0^t
h(I_s^{N,j})ds\Big) - \int_0^t h(I_s^{N,j}) ds\Big).
\end{align*}
In Theorem~\ref{T1}.1 we show convergence uniformly in the vertices of the graph, hence we need a result on convergence of $\frac 1NX^{N,i}$, uniform in $i\leq N$.
\begin{lemma}\label{l:poi2}
	It holds that
	\begin{align}
	\sup_{i\leq N}\sup_{0\leq s\leq t}\frac 1N \big\vert X^{N,i}_s\big\vert \to 0
	\end{align}
	almost surely (and in $L^2$).
\end{lemma}
\begin{proof}
	The sixth centered moment of a Poisson random variable with parameter $\lambda$ is of order $\lambda^3$. For each $i$, we may split $X^{N,i}$ up into its exciting and inhibiting part,
	\begin{align*}
		X^{N,i}_t = \sum_{j:U_j=1} V_{ji}&\Big(Y_j\Big(\int_0^t
		h(I_s^{N,j})ds\Big) - \int_0^t h(I_s^{N,j}) ds\Big) \\ &- \sum_{j:U_j=-1} V_{ji}\Big(Y_j\Big(\int_0^t
		h(I_s^{N,j})ds\Big) - \int_0^t h(I_s^{N,j}) ds\Big).
	\end{align*}
	This is the difference of two compensated point processes with intensity bounded by $N||h||$, hence the sixth moment of $X^{N,i}_t$ is at most of order $(tN||h||)^3$. We obtain, using Doob's martingale inequality for the martingale $X^{N,i}$,
	\begin{align*}
	\mathbf E\Big[\sup_{i\leq N}\sup_{0\leq s\leq t}\big(\frac 1N X^{N,i}_s\big)^6\Big]
	&\leq \sum_{i=1}^N \mathbf E\Big[\sup_{0\leq s\leq t}\big(\frac 1N X^{N,i}_s\big)^6\Big]
	\\ &\leq \Big(\frac 65\Big)^6 \sum_{i=1}^N \mathbf E\Big[\big(\frac 1N X^{N,i}_t\big)^6\Big]
	= O(N^{-2}),
	\end{align*}
	which is summable. Using Borel-Cantelli, almost-sure convergence follows.
\end{proof}
\noindent
In Theorem 1.3, we scale up by $\sqrt{N}$ and investigate on the correlation structure at different vertices, hence we need a result on joint convergence of $\big(\frac 1{\sqrt{N}}X^{N,i}\big)_{i=1,...,n}$. Therefore we introduce the processes
\begin{align*}
X^{N,0}_t & = \sum_{j=1}^N \Big(Y_j\Big(\int_0^t
h(I_s^{N,j})ds\Big) - \int_0^t h(I_s^{N,j}) ds\Big),
\\ X^{N,U}_t & = \sum_{j=1}^N U_j\Big(Y_j\Big(\int_0^t
h(I_s^{N,j})ds\Big) - \int_0^t h(I_s^{N,j}) ds\Big),
\\ \widetilde X^{N,i}_t & = \sum_{j=1}^N U_j(V_{ji}-q)\Big(Y_j\Big(\int_0^t
h(I_s^{N,j})ds\Big) - \int_0^t h(I_s^{N,j}) ds\Big),
\intertext{and split}
X^{N,i}_t &=\widetilde X^{N,i}_t+qX^{N,U}_t,\text{   for }i=1,2,...
\end{align*}
Note that only the first summand in the latter equation depends on the vertex $i$.
\begin{lemma}\label{l:donsker}
	Assume $\sup_{i\leq N}\sup_{s\leq t}|I^{N,i}_s-I_s| \to 0$ almost surely. Then 
	\begin{align*}
	\frac 1{\sqrt{N}}X^N=\frac 1{\sqrt{N}}\big( X^{N,U},X^{N,0},\widetilde X^{N,1},...,\widetilde X^{N,n}\big) \xRightarrow{N\to\infty} \big( M^U,M^{0},\widetilde M^{1},...,\widetilde M^{n}\big)=:M,
	\end{align*}
	where $M^U,M^{0},\widetilde M^{1},...,\widetilde M^{n}$ are local martingales with covariance structure given by 
	\begin{equation}
	\begin{aligned}\label{eq:cov}
	d\langle \widetilde M^k,\widetilde M^l \rangle_t &= 1_{k=l}q(1-q)h(I_t)dt, \qquad &d\langle M^U,M^U \rangle_t = d\langle M^0,M^0 \rangle_t = h(I_t)dt,
	\\d\langle \widetilde M^k,M^U \rangle_t &= d\langle \widetilde M^k,M^0 \rangle_t = 0,
	&d\langle M^U,M^0 \rangle_t = (2p-1)h(I_t)dt.
	\end{aligned}
	\end{equation}
	for any $k,l\neq 0$. We can extend the probability space such that the convergence is almost surely and in $L^2$, uniformly on compact time intervals.
\end{lemma}
\noindent 
We further extend the probability space by Brownian motions $B^U, B^0, \widetilde B^i$, such that
\[M^U_t=\int_0^t\sqrt{h(I_s)}dB^U_s,\quad M^0_t=\int_0^t\sqrt{h(I_s)}dB^0_s,\quad\widetilde M^k_t=\int_0^t\sqrt{q(1-q)h(I_s)}d\widetilde B^k_s.\]
Necessarily, $\langle \widetilde B^k,\widetilde B^l \rangle_t=1_{k=l}t$, $\langle B^U,B^0 \rangle_t=(2p-1)t$ and $\langle \widetilde B^k,B^U \rangle_t=\langle \widetilde B^k,B^0 \rangle_t=0$ for $k,l=1,...,n$. Finally observe that, on this probability space,
\begin{align}\label{eq:brownian}
\frac 1{\sqrt{N}}X^{N,k} \to \widetilde M^k+qM^U = q\int_0^\cdot \sqrt{h(I_s)}B^U_s + \sqrt{q(1-q)}\int_0^\cdot \sqrt{h(I_s)}\widetilde B^k_s,
\end{align}
almost surely and in $L^2$ by the continuous mapping theorem.
\begin{proof}[Proof of Lemma \ref{l:donsker}]
	The result is a martingale central limit theorem, hence we need to identify the covariance structure. By \cite[Theorem I.4.52]{jacod2013limit}, 
	\begin{align*}
	\big[\frac 1{\sqrt{N}}\widetilde X^{N,k},\frac 1{\sqrt{N}}\widetilde X^{N,l}\big]_t &= \sum_{s\leq t}\sum_{j=1}^N\Delta Z^{N,j}_s \frac 1{\sqrt{N}}U_j(V_{jk}-q) \frac 1{\sqrt{N}}U_j(V_{jl}-q)
	\\ &= \sum_{j=1}^NZ^{N,j}_t \frac 1{\sqrt{N}}U_j(V_{jk}-q) \frac 1{\sqrt{N}}U_j(V_{jl}-q).
	\end{align*}
	The compensator of $Z^{N,j}$ is given by $\int_0^\cdot h(I^{N,j}_s)ds$, whence, by \cite[Proposition I.4.50 b)]{jacod2013limit},
	\begin{align*}
	\big\langle \frac 1{\sqrt{N}}\widetilde X^{N,k},\frac 1{\sqrt{N}}\widetilde X^{N,l} \big\rangle_t &= \sum_{j=1}^N \frac 1{\sqrt{N}}U_j(V_{jk}-q) \frac 1{\sqrt{N}}U_j(V_{jl}-q)\int_0^t h(I^{N,j}_s)ds \\&\xrightarrow{N\to\infty} 1_{k=l}q(1-q)\int_0^t h(I_s)ds.
	\end{align*}
	Analogously we obtain
	\begin{align*}
	\langle \frac 1{\sqrt{N}}\widetilde X^{N,U},\frac 1{\sqrt{N}}\widetilde X^{N,U} \rangle_t &= \sum_{j=1}^N \frac 1{\sqrt{N}}U_j \frac 1{\sqrt{N}}U_j\int_0^t h(I^{N,j}_s)ds \to \int_0^t h(I_s)ds
	\\ \langle \frac 1{\sqrt{N}}\widetilde X^{N,0},\frac 1{\sqrt{N}}\widetilde X^{N,U} \rangle_t &= \sum_{j=1}^N \frac 1{\sqrt{N}} \frac 1{\sqrt{N}}U_j\int_0^t h(I^{N,j}_s)ds \to (2p-1)\int_0^t h(I_s)ds
	\\ \langle \frac 1{\sqrt{N}}\widetilde X^{N,0},\frac 1{\sqrt{N}}\widetilde X^{N,0} \rangle_t &= \sum_{j=1}^N \frac 1{\sqrt{N}} \frac 1{\sqrt{N}}\int_0^t h(I^{N,j}_s)ds \to \int_0^t h(I_s)ds
	\\ \langle \frac 1{\sqrt{N}}\widetilde X^{N,k},\frac 1{\sqrt{N}}\widetilde X^{N,0} \rangle_t &= \sum_{j=1}^N \frac 1{\sqrt{N}}U_j(V_{jk}-q) \frac 1{\sqrt{N}}\int_0^t h(I^{N,j}_s)ds \to 0
	\\ \langle \frac 1{\sqrt{N}}\widetilde X^{N,k},\frac 1{\sqrt{N}}\widetilde X^{N,U} \rangle_t &= \sum_{j=1}^N \frac 1{\sqrt{N}}U_j(V_{jk}-q) \frac 1{\sqrt{N}}U_j\int_0^t h(I^{N,j}_s)ds \to 0.
	\end{align*}
	The desired convergence in distribution follows from \cite[Theorem VIII.3.8, b)(ii)$\to$(i)]{jacod2013limit}: $[\sup -\beta_5]$ holds as the first characteristic is 0, and $[\hat{\delta}_5-\mathbb R]$ is obvious as the jumps of $\widetilde X^{N,k},X^{N,U},X^{N,0}$ have size $1/{\sqrt{N}}$. Finally, condition $[\gamma_5-\mathbb R]$ is precisely the convergence of the predictable quadratic covariation above. 
	\\By Skorohod's Theorem, we can extend our probability space such that this convergence is almost surely with respect to Skorohod distance, and by continuity of the limit it is equivalent to local uniform convergence. For $L^2$ convergence, observe that $ X^{N,U},X^{N,0},\widetilde X^{N,k}$ are compensated point processes with intensity bounded by $N||h||$. Let $Y$ be a unit rate Poisson process and $B$ be a standard Brownian motion, such that $M^U_t=B_{\int_0^th(I_s)ds}$(\cite[Theorem~16.4]{b2Kallenberg2002}). Then uniform integrability follows from the computation
	\begin{align*}
	\mathbf E\Big[\sup_{0\leq s\leq t} \Big(\frac{1}{\sqrt{N}}X^{N,U}_s - M^U_s\Big)^4\Big]^{1/4} 
	 & \leq C
	\mathbf E\Big[\Big(\frac{1}{\sqrt{N}}X^{N,U}_t - M^U_t\Big)^4\Big]^{1/4}
	\\& \leq C\Big(	\mathbf E\Big[\Big(\frac{1}{\sqrt{N}} X^{N,U}_t\Big)^4\Big]^{1/4} + \mathbf E\Big[\Big(M^U_t\Big)^4\Big]^{1/4}\Big)
	\\ & \leq C \Big(\mathbf E\Big[\sup_{0\leq s\leq t||h||} \Big(\frac{1}{\sqrt{N}}(Y(Ns)-Ns)\Big)^4\Big]^{1/4}
	+ \mathbf E\Big[\sup_{0\leq s\leq t||h||}(B_s)^4\Big]^{1/4}\Big)
	\\ & \leq C \Big(\mathbf E\Big[\Big(\frac{1}{\sqrt{N}}(Y(Nt||h||)-Nt||h||)\Big)^4\Big]^{1/4}
	+ \mathbf E\Big[(B_{t||h||})^4\Big]^{1/4}\Big)
	\\ & = C\Big( \big(3(t||h||)^2 + t||h||/N\big)^{1/4} + (3(t||h||)^2)^{1/4} \Big) ,
	\end{align*}
	where we used Doob's and Minkovski's inequalities. In the third inequality, we used the time-change representation of point processes (see e.g.\ Chapter~6 of \cite{EthierKurtz86}), which we already used in Lemma~\ref{l:reform}.
\end{proof}

\subsection{Results on the synaptic weights}
\label{ss:weights}
In this section we collect results on different types of mean values of $U_j, V_{ji}$. We start with laws of large numbers, which we need to derive mean-field limits.
\begin{lemma}\label{l:LLN}
	It holds that 
	\begin{align}
	\label{eq:LLN1}
	\sup_{j\leq N}|\frac{1}{N} \sum_{i=1}^N V_{ji} - q| & \xrightarrow{N\to\infty} 0,
	\\ \label{eq:LLN2}
	\sup_{j\leq N}|\frac{1}{N} \sum_{i=1}^N U_iV_{ji} - (2p-1)q| & \xrightarrow{N\to\infty} 0,
	\\ \label{eq:LLN3}
	\sup_{j\leq N}|\frac{1}{N} \sum_{i=1}^N U_iV_{ji}V_{ik} - (2p-1)q^2| & \xrightarrow{N\to\infty} 0,
	\end{align}
	almost surely.
\end{lemma}

\begin{proof}
	For \eqref{eq:LLN1}, observe that the sixth central moment of a binomial distribution with parameters $N$ and $p$ is of order $N^3$, whence
	\begin{equation}
	\begin{aligned}\label{eq:LLNunif}
	\mathbf P\Big( \sup_{j\leq N} |\frac{1}{N} \sum_{i=1}^N V_{ji}-q|>\varepsilon\Big) & \leq \frac 1{\varepsilon^6}\sum_{j\leq N} \mathbf E\Big(\frac{1}{N} \sum_{i=1}^N V_{ji}-q \Big)^6 
	\\ & = \frac 1{\varepsilon^6N^5} \mathbf E\Big(\sum_{i=1}^N V_{1i}-q\Big)^6  = O(N^{-2}),
	\end{aligned}
	\end{equation}
	is summable. Uniform convergence follows from the Borel-Cantelli Lemma. Similarly in \eqref{eq:LLN2} and \eqref{eq:LLN3}.
\end{proof}

\noindent
In order to derive results on the fluctuation around the mean-field limit, we need the following central limit theorems. First define $W^N=\frac{1}{\sqrt{N}}\sum_{j=1}^N(U_j-(2p-1))$ and $\widetilde W^{N,i}=\frac{1}{\sqrt{N}}\sum_{j=1}^NU_j(V_{ji}-q)$, such that 
\begin{align*}
W^{N,i} := \frac{1}{\sqrt{N}}\sum_{j=1}^N(U_jV_{ji}-(2p-1)q) = qW^N + \widetilde W^{N,i}.
\end{align*}

\begin{lemma}\label{l:CLT}
	For any $n\in\mathbb N$ it holds that
	\begin{align}\label{eq:weightCLT}
	\big(W^N, \widetilde W^{N,1},...,\widetilde W^{N,n}\big) \xRightarrow{N\to\infty} \big(W, \sqrt{q(1-q)}\widetilde W^{1},..., \sqrt{q(1-q)}\widetilde W^{n}\big),
	\end{align}
	where $W\sim N(0,4p(1-p))$, $\widetilde W^{i}\sim N(0,1)$ and $W, \widetilde W^{1}, \widetilde W^{2},...$ are independent. Furthermore, for any $i\in\{1,...,n\}$, 
	\begin{align}\label{eq:CLT}
	W^{N,i} \xRightarrow{N\to\infty} qW+\sqrt{q(1-q)}\widetilde W^i,
	\end{align}
	as well as
	\begin{align}
	\label{eq:CLT1}
	\overline W^N &:=\frac 1N \sum_{i=1}^N W^{N,i} = qW^N + \frac 1N \sum_{i=1}^N \widetilde W^{N,i}\xRightarrow{N\to\infty} qW,
	\\ \label{eq:CLT2}
	\overline W^{N,U} &:=\frac 1N \sum_{i=1}^N U_iW^{N,i} = qW^N\frac 1N \sum_{i=1}^N U_i + \frac 1N \sum_{i=1}^NU_i \widetilde W^{N,i}\xRightarrow{N\to\infty} (2p-1)qW,
	\\ \label{eq:CLT3}
	\overline W^{N,UVk} &:=\frac 1N \sum_{i=1}^N U_iV_{ik}W^{N,i} = qW^N\frac 1N \sum_{i=1}^N U_iV_{ik} + \frac 1N \sum_{i=1}^NU_iV_{ik} \widetilde W^{N,i}\xRightarrow{N\to\infty} (2p-1)q^2W,
	\end{align}
	for any $k\in\mathbb N$.
\end{lemma}

\begin{proof}
	First, \eqref{eq:weightCLT} follows from the central limit theorem, as 
	\begin{align*}
	\mathbf E[(U_j-(2p-1))^2]&=4p(1-p),\\
	\mathbf E[U_j^2(V_{ji}-q)^2]=\mathbf E[(V_{ji}-q)^2]&=q(1-q),
	\end{align*}
	and the summands of $W^N,\widetilde W^{N,1},...,\widetilde W^{N,n}$, respectively, are uncorrelated. Then \eqref{eq:CLT} follows by definition of $W^N$ and the continuous mapping theorem. Next observe that, for any $i,j,k,l,m\in\{1,...,N\}$, $k\neq l$,
	\begin{align*}
	\mathbf E[U_i(V_{ik}-q)U_j(V_{jl}-q)]&=0,
	\intertext{and hence}
	\mathbf E[\widetilde W^{N,k}\widetilde W^{N,l}]&=0,\\
	\mathbf E[U_k\widetilde W^{N,k}U_l\widetilde W^{N,l}]&=0,\\
	\mathbf E[U_kV_{km}\widetilde W^{N,k}U_lV_{lm}\widetilde W^{N,l}]&=0.
	\end{align*}
	Consequently the second summand in \eqref{eq:CLT1}, \eqref{eq:CLT2}, \eqref{eq:CLT3}, respectively, converges to $0$ in probability by the weak law of large numbers. The convergence of the first summands follows from convergence of $W^N$ and lemma \ref{l:LLN}
\end{proof} 

\noindent
Assume from now on that we work on a probability space where $W^N\to W$ almost surely. Then the convergence in \eqref{eq:CLT1}, \eqref{eq:CLT2}, \eqref{eq:CLT3} is in probability, and we obtain the following
\begin{lemma}\label{l:CLTmix}
	Assume $W^N\to W$ almost surely. Then
	\begin{align}\label{eq:squaremix}
	\frac 1N\sum_{i=1}^N\big(W^{N,i}\big)^2 \to_p q^2W^2 + q(1-q),
	\end{align}
	where $\to_p$ denotes convergence in probability.
\end{lemma}
\begin{proof}
	We decompose
	\begin{align}
	\frac 1N\sum_{i=1}^N\big(W^{N,i}\big)^2 &= q^2\underbrace{\big(W^N\big)^2}_{\to_{as} W^2} + q\underbrace{W^N}_{\to_{as} W}\underbrace{\frac 1N\sum_{i=1}^N\widetilde W^{N,i}}_{\to_p 0}+\underbrace{\frac 1N\sum_{i=1}^N\big(\widetilde W^{N,i}\big)^2}_{\to_p q(1-q)}\to_p q^2W^2 + q(1-q).
	\end{align}
	The convergence $\frac 1N\sum_{i=1}^N\widetilde W^{N,i}\to 0$ is in probability, as $(\widetilde W^{N,i})_i$ are uncorrelated. For the convergence in the last summand, observe that $\big(\widetilde W^{N,i}\big)^2$ are independent given $(U_j)_j$. Denote by $\mathbf{E}_{U}$ conditional expectation with respect to $(U_j)_j$. As $\mathbf E_U\big[\big(\widetilde W^{N,i}\big)^2\big]=q(1-q)$, it holds that 
	\begin{align*}
	\mathbf E_U\Big[\frac 1N\sum_{i=1}^N\big(\widetilde W^{N,i}\big)^2-q(1-q)\Big]\to 0,
	\end{align*}
	almost surely. It follows that
	\begin{align*}
	\mathbf E_U\Big[\Big(\frac 1N\sum_{i=1}^N\big(\widetilde W^{N,i}\big)^2-q(1-q)\Big)\wedge 1\Big]\to 0,
	\end{align*}
	and, by dominated convergence, the desired convergence in probability, \[\mathbf E\Big[\Big(\frac 1N\sum_{i=1}^N\big(\widetilde W^{N,i}\big)^2-q(1-q)\Big)\wedge 1\Big]\to 0.\]
\end{proof}

\subsection{Proof of Theorem~\ref{T1} and Corollary~\ref{cor1}}
\label{ss:proofT1}

In the following, we work on a probability space where the convergence in Lemma \ref{l:donsker} and in \eqref{eq:weightCLT} is in probability. The constant $C$ depends on $h,\varphi,t,p,q$ and may change from line to line. The proof of Theorem 1 is based on appropriate decompositions of the relevant processes, where one part vanishes as $N\to\infty$, and we pbtain convergence using Gronwall's inequality. We first collect some results on convergence and boundedness of these sub-processes in Lemma~\ref{l:t1prelim}. Let $I$ be the unique solution to \eqref{eq:IE} and $I^{N,i}$ as in \eqref{eq:IN} in Theorem~1. Define 
\begin{align*}
A^{N,i}_t  & = \frac 1N\sum_{j=1}^N U_jV_{ji} \Big( Y_j\Big(\int_0^t
h(I_s^{N,j})ds\Big) - \int_0^t h(I_s^{N,j}) ds\Big),
\\ D^{N,i}_t & = \frac 1N\sum_{j=1}^N \big( U_jV_{ji} - (2p-1)q \big) \int_0^t h(I_s) ds.
\end{align*}
Recall $B^U$ from Lemma~\ref{l:donsker}, $W$ from Lemma~\ref{l:CLT} and define
\begin{align*}
\mathcal A^{N}_t & = \frac1N\sum_{i=1}^NU_i\frac 1{\sqrt{N}}\sum_{j=1}^N U_jV_{ji} \Big(Y_j\Big(\int_0^t
h(I_s^{N,j})ds\Big) - \int_0^t h(I_s^{N,j}) ds \Big)
\\ &\qquad\qquad\qquad\qquad - (2p-1)q\int_0^t \sqrt{h(I_s)} dB^U_s,\\
\mathcal C_t^{N} & =  \frac1N\sum_{i=1}^NU_i\frac 1N \sum_{j=1}^N U_jV_{ji} \int_0^t \sqrt{N}\big(h(I_s^{N,j}) - h({I}_s)-(I_s^{N,j} - {I}_s)h'(I_s)\big) ds,\\
\mathcal D_t^{N} & = \Big( \frac1N\sum_{i=1}^NU_i\frac 1{\sqrt N} \sum_{j=1}^N (U_jV_{ji}-(2p-1)q) -(2p-1)qW \Big)\int_0^t h(I_s)ds,\\
\mathcal E_t^{N} & =  \frac1N\sum_{j=1}^NU_j\Big(\frac 1N \sum_{i=1}^N U_iV_{ji}-(2p-1)q\Big) \int_0^t \sqrt{N} (I_s^{N,j} - {I}_s)h'(I_s)ds.
\end{align*}
Further, recall $\widetilde B^i$ from Lemma~\ref{l:donsker}, $\widetilde W^i$ from Lemma~\ref{l:CLT} and define
\begin{align*}
\widetilde {\mathcal A}^{N,i}_t & = \frac 1{\sqrt{N}}\sum_{j=1}^N U_jV_{ji} \Big(Y_j\Big(\int_0^t
h(I_s^{N,j})ds\Big) - \int_0^t h(I_s^{N,j}) ds \Big)
\\ &\qquad\qquad\qquad\qquad -q\int_0^t \sqrt{h(I_s)}B^U_s + \sqrt{q(1-q)}\int_0^t \sqrt{h(I_s)}\widetilde B^k_s\\
\widetilde {\mathcal C}_t^{N,i} & =  \frac 1N \sum_{j=1}^N U_jV_{ji} \int_0^t \sqrt{N}\big(h(I_s^{N,j}) - h({I}_s)-(I_s^{N,j} - {I}_s)h'(I_s)\big) ds,\\
\widetilde {\mathcal D}_t^{N,i} & = \Big( \frac 1{\sqrt N} \sum_{j=1}^N \big(U_jV_{ji}-(2p-1)q\big)-\big(qW+\sqrt{q(1-q)}\widetilde W^i\big)\Big)\int_0^t h(I_s)ds.
\end{align*}
Denote by $\mathbf{E}_{UV}$ conditional expectation with respect to the configuration of the graph $(U_j)_j, (V_{ji})_{ji}$.	
\begin{lemma}\label{l:t1prelim}
For any $t>0$, it holds that
\begin{enumerate}
	\item[1.+2.] $\sup_{i\leq N}\sup_{s\leq t}A^{N,i}_s \xrightarrow{N\to\infty} 0$ and $\sup_{i\leq N}\sup_{s\leq t}D^{N,i}_s \xrightarrow{N\to\infty} 0$ almost surely,
\end{enumerate}
Assume $\sup_{i\leq N}\sup_{s\leq t}|I^{N,i}_s-I_s| \xrightarrow{N\to\infty} 0$, then
\begin{enumerate}
	\item[3.+4.] $\limsup_N\sup_{i\leq N}\sup_{0\leq s\leq t}\sqrt{N}A_s^{N,i}$ and $\limsup_N\sup_{i\leq N}\sup_{0\leq s\leq t}\sqrt{N}D^{N,i}$ are finite almost surely,
	\item[5.] $\limsup_N\frac1N\sum_{j=1}^N\mathbf{E}_{UV}\big[\sup_{0\leq s\leq t}\big(\sqrt{N} (I_s^{N,j} - {I}_s)\big)^2\big]$ is finite almost surely.
\end{enumerate}
We further have
\begin{enumerate}
	\item[6.-9.] $\mathbf{E}_{UV}\big[\sup_{0\leq s\leq t}\big(\mathcal A^{N}_s\big)^2\big]$, $\mathbf{E}_{UV}\big[\sup_{0\leq s\leq t}\big(\mathcal C^{N}_s\big)^2\big]$, $\mathbf{E}_{UV}\big[\sup_{0\leq s\leq t}\big(\mathcal D^{N}_s\big)^2\big]$, $\mathbf{E}_{UV}\big[\sup_{0\leq s\leq t}\big(\mathcal E^{N}_s\big)^2\big]$ converge to 0 in probability, as $N\to\infty$,
\end{enumerate}
and
\begin{enumerate}
	\setcounter{enumi}{9}
	\item $\mathbf{E}_{UV}\big[\sup_{0\leq s\leq t}\big(\widetilde {\mathcal A}^{N,i}_s\big)^2\big] \xrightarrow{N\to\infty} 0$ in probability, for any $i=1,2,...$.
	\item $\sup_{i\leq N}\mathbf{E}_{UV}\big[\sup_{0\leq s\leq t}\big(\widetilde {\mathcal C}^{N,i}_s\big)^2\big] \xrightarrow{N\to\infty} 0$ in probability,
	\item $\mathbf{E}_{UV}\big[\sup_{0\leq s\leq t}\big(\widetilde {\mathcal D}^{N,i}_s\big)^2\big] \xrightarrow{N\to\infty} 0$ in probability, for any $i=1,2,...$.
\end{enumerate}
\end{lemma}
We use 1 and 2 in the proof of Theorem~1.1, and 3, 4 and 11 to prove 5. Then, we use 5 to prove 9, 6-9 in Step 1 of the proof of Theorem~1.3, and 10-12 in Step 3 of the proof of Theorem~1.3. We give the proof of Theorem~\ref{T1} first, the proof of Lemma~\ref{l:t1prelim} can be found at the end of this section.

\begin{proof}[Proof of Theorem \ref{T1}]
	{\bf Proof of 1.:}\\
	Define
	\begin{align}\label{eq:decomposition}
	J_t = (2p-1)q\int_0^t h(I_s) ds, \text{ such that } I_t = \int_0^t \varphi(t-s)dJ_s.
	\end{align}
	Recall $J^N$ from \eqref{eq:JN}. For $\theta_N = \tfrac{1}{N}$, with $A^{N,i}, D^{N,i}$ as in Lemma~\ref{l:t1prelim} and $C^{N,i}_t = \frac 1N\sum_{j=1}^N U_jV_{ji} \int_0^t h(I_s^{N,j}) - h(I_s) ds$, we get that,
	\begin{align*}
	J_t^{N,i} - J_t & = A^{N,i} + C^{N,i} + D^{N,i},
	\end{align*}
	We first show that $\sup_{s\leq t}\frac 1N\sum_{j=1}^N (J_s^{N,i} - J_s)^2 \to 0$ almost surely. Observe that
	\begin{align}\label{eq:prev}
	\sup_{s\leq t} \frac 1N\sum_{i=1}^N \big(J_s^{N,i} - J_s\big)^2 & \leq \frac 3N\sum_{i=1}^N \Big(\sup_{s\leq t}(A^{N,i}_s)^2 + \sup_{s\leq t}(C^{N,i}_s)^2 + \sup_{s\leq t}(D^{N,i}_s)^2\Big)
	\end{align}
	By Lemma~\ref{l:t1prelim}.1 and Lemma~\ref{l:t1prelim}.2, respectively, we have $\sup_{i\leq N}\sup_{s\leq t}(A^{N,i}_s)^2 \to 0$ and $\sup_{i\leq N}\sup_{s\leq t}(D^{N,i}_s)^2 \to 0$ almost surely. For $C^{N,i}$ we obtain, using Jensen's inequality twice and $U_j^2=1$, $V_{ji}^2\leq 1$,
	\begin{equation}\label{eq:meanC}
	\begin{aligned}
	\sup_{s\leq t} \frac 1N\sum_{i=1}^N\Big(C^{N,i}_s\Big)^2 & \leq \sup_{s\leq t} \frac 1N\sum_{i=1}^N\frac 1N\sum_{j=1}^N \Big(\int_0^t h(I_s^{N,j}) - h(I_s) ds\Big)^2
	\\ & \leq t \frac 1N\sum_{j=1}^N \int_0^t \big(h(I_s^{N,j}) - h(I_s)\big)^2 ds
	\\ & \leq t h_{Lip}^2 \frac 1N\sum_{j=1}^N \int_0^t \big(I_s^{N,i} - I_s\big)^2 ds 
	\\ & \leq C t h_{Lip}^2 \int_0^t \sup_{u\leq s}\frac 1N\sum_{j=1}^N(J_u^{N,i} - J_u)^2 ds,
	\end{aligned}
	\end{equation}
	where we have used Lemma \ref{l:conv} in the last step. Combining the results on $A^{N,i}, C^{N,i}$ and $D^{N,i}$ in \eqref{eq:prev} we obtain
	\begin{align*}
	\sup_{s\leq t} \frac 1N\sum_{j=1}^N\big(J_s^{N,i} - J_s\big)^2 \leq 3Ct h_{lip}^2 \int_0^t \sup_{u\leq s}\frac 1N\sum_{j=1}^N\big(J_u^{N,i} - J_u\big)^2 ds + o(1),
	\end{align*}
	hence $\sup_{s\leq t} \frac 1N\sum_{i=1}^N (J_s^{N,i} - J_s)^2 \to 0$ by Gronwall's inequality. Now fix $i$ and repeat the estimation in \eqref{eq:meanC} to obtain 
	\begin{align*}
	\sup_{s\leq t} \Big(C^{N,i}_s\Big)^2 & \leq \sup_{s\leq t} \frac 1N\sum_{j=1}^N \Big(\int_0^s h(I_u^{N,j}) - h(I_u) ds\Big)^2 
	\\ & \leq C t h_{Lip}^2 \int_0^t \sup_{u\leq s}\frac 1N\sum_{j=1}^N(J_u^{N,j} - J_u)^2 ds \to 0
	\end{align*}
	almost surely. As the right hand side does not depend on $i$, this convergence is uniform in $i\leq N$. As we have already shown $\sup_{i\leq N}\sup_{s\leq t}(A^{N,i}_s)^2\to 0$ as well as $\sup_{i\leq N}\sup_{s\leq t}(D^{N,i}_s)^2\to 0$, we obtain 
	\begin{align*}
	\sup_{i\leq N}\sup_{s\leq t} \big(J_s^{N,i} - J_s\big)^2 \leq 3 \Big( \sup_{i\leq N}\sup_{s\leq t}(A^{N,i}_s)^2 + \sup_{i\leq N}\sup_{s\leq t}(C^{N,i}_s)^2 + \sup_{i\leq N}\sup_{s\leq t}(D^{N,i}_s)^2 \Big) \to 0
	\end{align*}
	almost surely. By Lemma \ref{l:conv}, we can conclude that $\sup_{i\leq N}\sup_{s\leq t} \big(I^{N,i}_s - I_s\big)^2 \to 0$ almost surely.\\
	
	\noindent
	{\bf Proof of 2.:}\\
	Define $\bar Z^i_t = Y_i\big( \int_0^t h(I_s)ds\big)$, where $Y_1,...,Y_N$ are independent Poisson process as in \eqref{eq:timechange0}. Fix $\omega\in\Omega$, such that \[\sup_{i\leq N}\sup_{0\leq s\leq t} | I_s^{N,i}(\omega) - I_s | \rightarrow 0.\]
	As $h$ is Lipschitz, $\int_0^t h(I_s^{N,i}(\omega))ds \rightarrow \int_0^t h(I_s)ds$, hence for any point of continuity of $t\mapsto \bar Z^i_t(\omega)$ we can conclude $Z^i_t(\omega) \rightarrow \bar Z^i_t(\omega)$. As the points of continuity are dense in $[0,\infty)$, convergence in Skorohod-distance follows from \cite[Theorem~2.15~c)(ii)]{jacod2013limit}.\\ 
	
	\noindent
	{\bf Proof of 3.:}\\
	First, strong existence and uniqueness of $\overline K$ follows from \cite{berger1980volterra}. For the convergence result, we proceed in three steps:\\
	
	\noindent
	{\bf Step 3.1: $\frac1N\sum_{i=1}^NU_i\sqrt{N}(I^{N,i}-I)\to (2p-1)\overline{K}$}\\[0.5em]
	Use $B^U$ from Lemma \ref{l:donsker} in the definition of $\overline{K}$ and define $\overline{K}^U:=(2p-1)\overline{K}$. Then $\overline{K}^U=\int_0^t\varphi(t-s)d\overline{G}^U_s$ with
	\begin{align*}
	\overline{G}^U=\int_0^t (2p-1)qW h(I_s) + (2p-1)qh'(I_s)\overline{K}^U_s ds + (2p-1)q \int_0^t \sqrt{h(I_s)} dB^U_s.
	\end{align*}
	We show $\frac1N\sum_{i=1}^NU_i\sqrt{N}(I^{N,i}-I)\to \overline{K}^U$ in probability, uniformly on compact time intervals. With $\mathcal A^{N}$, $\mathcal C^{N}$, $\mathcal D^{N}$, $\mathcal E^{N}$ as in Lemma~\ref{l:t1prelim} and 
	\[
	\mathcal F_t^{N} := (2p-1)q \int_0^t\Big(\frac1N\sum_{j=1}^NU_j\sqrt{N} (I_s^{N,j} - {I}_s)-\overline{K}^U_s\Big)h'(I_s)ds
	\]
	we can decompose
	\begin{align*}
	\frac1N\sum_{i=1}^NU_i\sqrt N (J^{N,i}_t-J_t) - \overline{G}^U_t = \mathcal A^{N}_t + \mathcal C^{N}_t+ \mathcal D^{N}_t+ \mathcal E^{N}_t + \mathcal F^{N}_t.  
	\end{align*}
	By Lemma~\ref{l:t1prelim}.6.-9., we have
	\[
	\mathbf{E}_{UV}\big[\sup_{0\leq s\leq t}\big(A^{N}_s\big)^2\big], \mathbf{E}_{UV}\big[\sup_{0\leq s\leq t}\big(C^{N}_s\big)^2\big], \mathbf{E}_{UV}\big[\sup_{0\leq s\leq t}\big(D^{N}_s\big)^2\big], \mathbf{E}_{UV}\big[\sup_{0\leq s\leq t}\big(E^{N}_s\big)^2\big]\to 0
	\]
	in probability. We now show 
	\[\mathbf{E}_{UV}\big[\sup_{0\leq s\leq t}\big(F^{N}_s\big)^2\big]\leq C\int_0^t\mathbf{E}_{UV}\big[\sup_{0\leq u\leq s}\big(\frac1N\sum_{i=1}^NU_i\sqrt N (J^{N,i}_u-J_u) - \overline{G}^U_u\big)^2\big]ds\]
	and deduce the desired convergence of $\frac1N\sum_{i=1}^NU_i\sqrt{N}(I^{N,i}-I)\to \overline{K}^U$ using Gronwall's inequality and Lemma~\ref{l:conv}. By Jensen's inequality and Lemma \ref{l:conv},
	\begin{align*}
	\mathbf{E}_{UV}\Big[\sup_{0\leq s\leq t}\big(F_s^{N}\big)^2\Big] & \leq C  \int_0^t \mathbf{E}_{UV}\Big[\sup_{0\leq u\leq s}\Big(\frac1N\sum_{j=1}^NU_j\sqrt{N} (I_u^{N,j} - {I}_u)-\overline{K}^U_u\Big)^2\Big]ds
	\\ &\leq C\int_0^t \mathbf{E}_{UV}\Big[\sup_{0\leq u\leq s}\Big(\frac1N\sum_{j=1}^NU_j\sqrt{N} (J_u^{N,j} - {J}_u)-\overline{G}^U_u\Big)^2\Big]ds.
	\end{align*}
	Hence we obtain 
	\begin{align*}
	\mathbf{E}_{UV}\Big[ \sup_{0\leq s\leq t}&\Big(\frac1N\sum_{i=1}^NU_i\sqrt N (J^{N,i}_s-J_s) - \overline{G}^U_s\Big)^2 \Big] \\
	&\leq o_p(1) + C\int_0^t \mathbf{E}_{UV}\Big[ \sup_{0\leq u\leq s}\Big(\frac1N\sum_{i=1}^NU_i\sqrt N (J^{N,i}_u-J_u) - \overline{G}^U_u\Big)^2\Big] ds,
	\end{align*}
	and by Gronwalls inequality
	\begin{align*}
	\mathbf{E}_{UV}\Big[ 1\wedge\sup_{0\leq s\leq t}&\Big(\frac1N\sum_{i=1}^NU_i\sqrt N (J^{N,i}_s-J_s) - \overline{G}^U_s\Big)^2 \Big] \to 0
	\end{align*}
	in probability. By bounded convergence, we have $\sup_{0\leq s\leq t}\Big(\frac1N\sum_{i=1}^NU_i\sqrt N (J^{N,i}_s-J_s) - \overline{G}^U_s\Big)^2\to 0$ in probability. \\[0.5em]
	
	\noindent
	{\bf Step 3.2: $\frac1N\sum_{i=1}^N\sqrt{N}(I^{N,i}-I)\to \overline{K}$ and $\frac1N\sum_{i=1}^NU_iV_{ik}\sqrt{N}(I^{N,i}-I)\to (2p-1)q\overline{K}$}\\[0.5em]
	From the uniform convergence $\frac1N\sum_{i=1}^NU_i\sqrt{N}(I^{N,i}-I)\to (2p-1)\overline{K}$ in probability we can easily deduce convergence of \[\frac1N\sum_{i=1}^N\sqrt{N}(I^{N,i}-I)\to \overline{K}\quad\text{ as well as }\quad\frac1N\sum_{i=1}^NU_iV_{ik}\sqrt{N}(I^{N,i}-I)\to (2p-1)q\overline{K}\]
	as follows: We can split $\frac1N\sum_{i=1}^N\sqrt{N}(I^{N,i}-I)\to \overline{K}$ in summands similar to $\mathcal A^N,...,\mathcal F^N$ from Step 3.1. Convergence of $\mathcal A^N,\mathcal C^N,\mathcal D^N,\mathcal E^N$ follows analogously to Lemma~\ref{l:t1prelim}, we simply use \eqref{eq:LLN1} instead of \eqref{eq:LLN2} for $\mathcal A^N,\mathcal E^N$ and \eqref{eq:CLT1} instead of \eqref{eq:CLT2} for $\mathcal D^N$. Convergence of $\mathcal F^N$ follows from the uniform convergence of \[\frac1N\sum_{i=1}^NU_i\sqrt{N}(I^{N,i}-I)\to (2p-1)\overline{K}\]
	from Step 3.1. For $\frac1N\sum_{i=1}^NU_iV_{ik}\sqrt{N}(I^{N,i}-I)$, use \eqref{eq:LLN3} instead of \eqref{eq:LLN2} and \eqref{eq:CLT3} instead of \eqref{eq:CLT2}.\\
	
	\noindent
	{\bf Step 3.3: $\sqrt{N}(I^{N,i}-I)\to {K^i}$}\\[0.5em]
	Use $\widetilde B^i$ from Lemma \ref{l:donsker} in the definition of $G^i$ and recall $\widetilde W^i$ from \ref{ss:weights}. With $\widetilde {\mathcal A}^{N,i}$, $\widetilde {\mathcal C}^{N,i}$ and $\widetilde {\mathcal D}^{N,i}$ from Lemma~\ref{l:t1prelim},
	\begin{align*}
	\sqrt{N}(J^{N,i}-J)-G^i &= \widetilde {\mathcal A}^{N,i}_t + \widetilde {\mathcal C}^{N,i}_t+ \widetilde {\mathcal D}^{N,i}_t+ \widetilde {\mathcal E}^{N,i}_t  
	\intertext{where}
	\widetilde {\mathcal E}_t^{N,i} & =  \int_0^t \Big(\Big(\frac 1N \sum_{j=1}^N U_jV_{ji}\sqrt{N} (I_s^{N,j} - {I}_s)\Big)-(2p-1)q\overline K\Big) h'(I_s)ds.
	\end{align*}
	With Lemma~\ref{l:t1prelim}.10-12 and step 3.2, we obtain $\sup_{0\leq s\leq t}\big(\sqrt{N}(J_s^{N,i}-J_s)\to {G_s^i}\big)^2\to0$ in probability. By Lemma \ref{l:conv}, $\sup_{0\leq s\leq t}\big(\sqrt{N}(I_s^{N,i}-I_s)\to {K_s^i}\big)^2\to0$ in probability. 
\end{proof}

\begin{proof}[Proof of Corollary \ref{cor1}]
	First, \eqref{eq:cor11} follows from the convergence of $\frac 1{\sqrt{N}}X^{N,0}$ in Lemma \ref{l:donsker}. For \eqref{eq:cor12}, by the law of large numbers, 
	\[\frac 1N \sum_{j=1}^N \bar Z^j_T \xrightarrow{N\to\infty} \int_0^T h(I_t) dt,\quad \text{as well as}\quad \frac 1N\sum_{i=1}^N h(I^{N,i}) - h(I)\xrightarrow{N\to\infty} 0\]
	almost surely, uniformly on compact time intervals by Theorem~\ref{T1}.1. Therefore
	\begin{align*}
	\frac 1N \sum_{j=1}^N\Big( Z^{N,j} - \bar Z^j\Big) & = 
	\frac 1N \sum_{j=1}^N\Big( Z^{N,j} - \int_0^. h(I_s^{N,i}) ds\Big) \\ & \qquad - 
	\frac 1N \sum_{j=1}^N\Big( \bar Z^{j} - \int_0^. h(I_s)ds\Big)
	+ \int_0^. \frac 1N\sum_{i=1}^N h(I^{N,i}) - h(I) ds \xrightarrow{N\to\infty} 0,
	\end{align*}
	almost surely. Next, \eqref{eq:cor13} follows from the reformulation \eqref{eq:timechange0}, Lemma \ref{l:donsker} and Theorem~\ref{T1}.1. For \eqref{eq:cor14}, with $K^i$ as in Theorem~\ref{T1}.3,
	\begin{align*}
	\sqrt N (h(I_t^{N,i}) - h(I_t)) =\sqrt N h'(I_t) (I_t^{N,i} - I_t) + o(1) = h'(I_t) K^i_t +  o(1).
	\end{align*}
	Last, for \eqref{eq:cor15}, recall $B, B^0$ from section \ref{ss:reformulation}. Use $B$ in the definition of $\overline K$, then $\mathbf E[B_t B^0_t] = (2p-1)qt, t\geq 0$ and
	\begin{align*}
	\frac{1}{\sqrt N} \sum_{j=1}^N \big(h(I^{N,j})-h(I)\big) = h'(I_s) \frac 1N\sum_{j=1}^N \sqrt{N}\big(I^{N,j}-I\big)+o(1)\xRightarrow{N\to\infty} h'(I_s)\overline K_s,
	\end{align*}
	by Theorem~\ref{T1}.3, whence
	\begin{align*}
	\frac{1}{\sqrt N} \sum_{j=1}^N \Big(Z^{N,j} - \int_0^t h(I_s)ds\Big) &= \frac{1}{\sqrt N} \sum_{j=1}^N \Big(Z^{N,j} - \int_0^t h(I^{N,j}_s)ds\Big) + \int_0^t\frac{1}{\sqrt N} \sum_{j=1}^N \big(h(I^{N,j})-h(I)\big)ds\\
	&\xRightarrow{N\to\infty} \int_0^t \sqrt{h(I_s)}dB^0_s + \int_0^t h'(I_s)\overline K_s ds 
	\end{align*}
	by \eqref{eq:cor13}.
\end{proof}

\begin{proof}[Proof of Lemma \ref{l:t1prelim}]
	{\bf Proof of 1.:} This is exactly the statement of Lemma~\ref{l:poi2}.\\[0.5em]
	{\bf Proof of 2.:} Follows from boundedness of $h$ and \eqref{eq:LLN2}.\\[0.5em]
	{\bf Proof of 3.:} By Lemma 4.51 and Proposition 4.50 b) in \cite{jacod2013limit}, the predictable quadratic variation of $\sqrt{N}A^{N,i}$ is given by $\frac 1N\sum_{i=1}^NV_{ji}\int_0^t h(I_s^{N,j})ds$. By Doob's inequality, the assumed uniform convergence of $I^{N,i}$ and boundedness of $h$,
	\begin{align*}
	\mathbf{E}_{UV}\Big[ \sup_{0\leq s\leq t}\big(\sqrt{N}A^{N,i}_s\big)^2 \Big]\leq C \mathbf{E}_{UV}\Big[ \langle \sqrt{N}A^{N,i}\rangle_t \Big] = \frac 1N\sum_{i=1}^NV_{ji}\int_0^t\mathbf{E}_{UV}\Big[h(I_s^{N,j})\Big]ds \to q\int_0^th(I_s)ds.
	\end{align*}\\[0.5em]
	{\bf Proof of 4.:} Recall $W^{N,j}$ from section \ref{ss:weights}. By boundedness of $h$, we have
	\begin{align*}
	\mathbf{E}_{UV}\Big[ \frac1N\sum_{j=1}^N\sup_{0\leq s\leq t}\big(\sqrt{N}D^{N,i}_s\big)^2 \Big] &\leq C \frac1N\sum_{j=1}^N\big(W^{N,j}\big)^2 
	\end{align*}
	which converges in probability by \eqref{eq:squaremix}.\\[0.5em]
	{\bf Proof of 5.:} First write
	\begin{align*}
	\sqrt N (J^{N,i}_s-J_s) &= \sqrt{N}A^{N,i}_t + \widetilde {\mathcal C}^{N,i}_t+ \sqrt{N}D^{N,i}_t+ \frac 1N \sum_{j=1}^N U_jV_{ji} \int_0^t \Big(\sqrt{N} (I_s^{N,j} - {I}_s)\Big) h'(I_s)ds.
	\end{align*}
	By 3., 4. and 11., $\sup_{i\leq N}\big(\sqrt{N}A^{N,i}_t + \widetilde {\mathcal C}^{N,i}_t+ \sqrt{N}D^{N,i}_t\big)$ is bounded, and for the last term we can compute
	\begin{align*}
	\mathbf{E}_{UV}\Big[ \sup_{0\leq s\leq t}\big(\frac 1N \sum_{j=1}^N U_jV_{ji} \int_0^s \Big(\sqrt{N} (I_u^{N,j} - {I}_u)\Big) &h'(I_u)du\big)^2 \Big] 
	\\&\leq C \int_0^t\mathbf{E}_{UV}\Big[\frac1N\sum_{j=1}^N\sup_{0\leq u\leq s}\big(\sqrt{N} (I_u^{N,j} - {I}_u)\big)^2\Big] ds.  
	\end{align*}
	Using Lemma~\ref{l:conv}, we obtain
	\begin{align*}
	\frac1N\sum_{j=1}^N\mathbf{E}_{UV}\big[\sup_{0\leq s\leq t}\big(\sqrt{N} (I_s^{N,j} - {I}_s)\big)^2\big] & \leq \mathcal O(1) + C \int_0^t\frac1N\sum_{j=1}^N\mathbf{E}_{UV}\Big[\sup_{0\leq u\leq s}\big(\sqrt{N} (I_u^{N,j} - {I}_u)\big)^2\Big] ds.
	\end{align*}
	The desired boundedness follows from Gronwall's inequality.\\[0.5em]
	{\bf Proof of 6.:} First,
	\begin{align*}
	\mathcal A^{N}_s  &= \frac 1{\sqrt{N}}\sum_{j=1}^N U_j\Big[\frac1N\sum_{i=1}^NU_iV_{ji}-(2p-1)q\Big]\Big(Y_j\Big(\int_0^t
	h(I_s^{N,j})ds\Big) - \int_0^t h(I_s^{N,j}) ds \Big)
	\\ &\qquad + (2p-1)q\frac 1{\sqrt{N}}\sum_{j=1}^N U_j\Big(Y_j\Big(\int_0^t
	h(I_s^{N,j})ds\Big) - \int_0^t h(I_s^{N,j}) ds \Big)- \int_0^t \sqrt{h(I_s)} dB^U_s.
	\end{align*}
	The predictable quadratic variation of the first line is given by 
	\[
	\frac 1N\sum_{j=1}^N \big[\frac1N\sum_{i=1}^NU_iV_{ji}-(2p-1)q\big]^2\int_0^t
	h(I_s^{N,j})ds
	\]
	which can be seen by Lemma 4.51 and Proposition 4.50 b) in \cite{jacod2013limit}, and converges to $0$ almost surely by the assumed uniform convergence of $I^{N,i}$ and \eqref{eq:LLN2}. The second line converges to $0$ by Lemma \ref{l:donsker}. Using Doob's inequality, we obtain
	\begin{align*}
	\mathbf{E}&_{UV}\Big[\sup_{0\leq s\leq t}\big(\mathcal A^{N}_s\big)^2\Big] \\& \leq C\,\mathbf{E}_{UV}\Big[\Big(\frac 1{\sqrt{N}}\sum_{j=1}^N U_j\Big[\frac1N\sum_{i=1}^NU_iV_{ji}-(2p-1)q\Big]\Big(Y_j\Big(\int_0^t
	h(I_s^{N,j})ds\Big) - \int_0^t h(I_s^{N,j}) ds \Big)\Big)^2\Big]
	\\ &\quad+ C\,\mathbf{E}_{UV}\Big[\Big(\frac 1{\sqrt{N}}\sum_{j=1}^N U_j\Big(Y_j\Big(\int_0^t
	h(I_s^{N,j})ds\Big) - \int_0^t h(I_s^{N,j}) ds \Big)- \int_0^t \sqrt{h(I_s)} dB^U_s\Big)^2\Big] \to 0
	\end{align*} 
	almost surely.\\[0.5em]
	{\bf Proof of 7.:} Note that $\mathcal C^{N,i} = \frac1N\sum_{i=1}^NU_i\widetilde {\mathcal C}^{N,i}$, whence $\big(\mathcal C^{N,i}\big)^2 \leq \frac1N\sum_{i=1}^N\big(\widetilde {\mathcal C}^{N,i}\big)^2$, and the result follows from 11.\\[0.5em]
	{\bf Proof of 8.:} Follows from boundedness of $h$ and \eqref{eq:CLT2}. \\[0.5em]
	{\bf Proof of 9.:} By 5. and the uniform convergence in \eqref{eq:LLN2},
	\begin{align}\label{eq:EN}
	\mathbf{E}_{UV}\Big[\sup_{0\leq s\leq t}\big(\mathcal E_s^{N}\big)^2\Big] & \leq C \sup_{j\leq N}\Big[\frac 1N \sum_{i=1}^N \big(U_iV_{ji}-(2p-1)q\big)\Big]^2 \int_0^t\frac1N\sum_{j=1}^N\mathbf{E}_{UV}\Big[\big(\sqrt{N} (I_s^{N,j} - {I}_s)\big)^2\Big]ds
	\end{align}
	converges to 0 in probability. \\[0.5em]
	{\bf Proof of 10.:} This is exactly the statement in \eqref{eq:brownian}, which is a consequence of Lemma~\ref{l:donsker}. \\[0.5em]
	{\bf Proof of 11.:} Write $R_1h$ for the first order remainder in Taylor's expansion of the function $h$. We use $U_j^2=1, V_{ji}^2\leq 1$, Jensen inequality, and some $\xi^{N,j}_s$ between $I^{N,j}_s$ and $I_s$
	\begin{align*}
	\mathbf{E}_{UV}\Big[\sup_{0\leq s\leq t}\big(\widetilde {\mathcal C}_s^{N,i}\big)^2\Big] & \leq \frac 1N \sum_{j=1}^N \int_0^t \mathbf{E}_{UV}\Big[\big(\sqrt N R_1h(I^{N,j}_s,I_s)\big)^2 ds\Big]\\
	&\leq t \frac 1N \sum_{j=1}^N\int_0^t N \mathbf{E}_{UV}\Big[(h'(\xi^{N,j}_s) - h'(I_s))^2 (I^{N,j}_s-I_s)^2\Big] ds \\
	&\leq t(h'_{Lip})^2\int_0^t \sum_{j=1}^N\mathbf{E}_{UV}\Big[ (I^{N,j}_s-I_s)^4\Big] ds.
	\end{align*}
    We obtain convergence to zero almost surely of the last term similarly to the Proof of Theorem 1.1: Recall the decomposition of $J^{N,i}-J=A^{N,i}+C^{N,i}+D^{N,i}$ from \eqref{eq:decomposition}. As the fourth centered moment of a Poisson distribution with parameter $\lambda$ is of order $\lambda^2$ and $h$ is bounded, we obtain $\sum_{j=1}^N\mathbf{E}_{UV}\big[ (A^{N,i})^4\big]=O(N^{-1})$ similarly to the proof of Lemma \ref{l:poi2}. 
     For $D^{N,i}$ rewrite
	\begin{align*}
	\sum_{i=1}^N\Big( \frac 1N\sum_{j=1}^N \big( U_jV_{ji} - (2p-1)q \big) \Big)^4 \leq \sup_{i\leq N}\Big( \frac 1N\sum_{j=1}^N \big( U_jV_{ji} - (2p-1)q \big) \Big)^2\frac 1N \sum_{j=1}^N \big(W^{N,i}\big)^2 \to 0
	\end{align*}
	in probability, by  uniform convergence in \eqref{eq:LLN2} and \eqref{eq:squaremix}. 
	We obtain, using $U_j^4,V_{ji}^4\leq 1$, Jensen's inequality and Lemma \eqref{l:conv},
	\begin{align*}
	\sum_{i=1}^N\mathbf{E}_{UV}\Big[ \sup_{s\leq t}(J^{N,i}_s-J_s)^4\Big] &\leq \sum_{i=1}^N\mathbf{E}_{UV}\Big[\sup_{s\leq t}\Big(\frac 1N\sum_{j=1}^N U_jV_{ji} \int_0^s h(I_u^{N,j}) - h(I_u) du\Big)^4\Big]+ o_p(1)
	\\ &\leq C \int_0^t\sum_{j=1}^N\mathbf{E}_{UV}\Big[ \sup_{u\leq s}(J^{N,j}_u-J_u)^4\Big]ds + o_p(1).
	\end{align*}
	Then $\sum_{j=1}^N\mathbf{E}_{UV}\Big[ \sup_{s\leq t}(I^{N,j}_s-I_s)^4\Big]\to 0$ in probability follows from Gronwall's inequality and again Lemma \eqref{l:conv}.\\[0.5em]
	{\bf Proof of 12.:} Follows from boundedness of $h$ and \eqref{eq:CLT}.
\end{proof} 
	
	\printbibliography
	
\end{document}